\documentclass[11pt]{article}
\usepackage{amssymb,amsmath,setspace,graphicx,multicol,wrapfig,amsfonts,amsthm,titling,url,array,parskip,enumerate}
\usepackage{hyperref}
\usepackage{placeins}
\usepackage{aurical,pbsi}
\usepackage[T1]{fontenc}
\usepackage[hmargin=2.5cm,vmargin=2.5cm]{geometry}

\makeatletter
\def\thm@space@setup{%
  \thm@preskip=0.1in
  \thm@postskip=0in
}
\makeatother

\numberwithin{equation}{section}

\theoremstyle{plain}
\newtheorem{thm}{Theorem}[section]
\newtheorem{lemma}[thm]{Lemma}

\theoremstyle{definition}
\newtheorem{defn}{Definition}[section]
\newtheorem{conj}{Conjecture}[section]

\newtheorem{rem}{Remark}[section]

\DeclareMathOperator{\wt}{wt}
\DeclareMathOperator{\weight}{weight}
\DeclareMathOperator{\Prob}{Prob}
\DeclareMathOperator{\type}{type}

\DeclareMathOperator{\D}{D}
\DeclareMathOperator{\E}{E}
\DeclareMathOperator{\A}{A}

\begin{document}

\begin{center} {\Large{\sc Multi-Catalan Tableaux and the Two-Species TASEP}} \\
\vspace{0.1in}
Olya Mandelshtam
\footnote{Phone: +1 (949) 689-5748\\
Fax: +1 (510) 642-8204\\
E-mail: olya@math.berkeley.edu}
 \end{center}
 
\begin{abstract}
The goal of this paper is to provide a combinatorial expression for the steady state probabilities of the two-species PASEP. In this model, there are two species of particles, one ``heavy'' and one ``light'', on a one-dimensional finite lattice with open boundaries. Both particles can swap places with adjacent holes to the right and left at rates 1 and $q$. Moreover, when the heavy and light particles are adjacent to each other, they can swap places as if the light particle were a hole. Additionally, the heavy particle can hop in and out at the boundary of the lattice. Our main result is a combinatorial interpretation for the stationary distribution at $q=0$ in terms of certain multi-Catalan tableaux. We provide an explicit determinantal formula for the steady state probabilities, as well as some general enumerative results for this case. We also describe a Markov process on these tableaux that projects to the two-species PASEP, and thus directly explains the connection between the two.  Finally, we give a conjecture that extends our formula for the stationary distribution to the $q=1$ case, using certain two-species alternative tableau.
\end{abstract}

 \section{Introduction}
The Partially Asymmetric Simple Exclusion Process (PASEP) is a well-studied model that describes the dynamics of particles hopping on a finite one-dimensional lattice on $n$ sites with open boundaries, with the rule that there is at most one particle in a site, and at most one particle hops at a time. Figure \ref{PASEP_parameters} shows the parameters of this process, with the Greek letters denoting the rates of the hopping particles. Processes of this flavor have been studied in many contexts, in particular for their connections to some very nice combinatorics.  For instance, see \cite{cw2007} and the references therein.

\begin{wrapfigure}[7]{r}{0.35\textwidth}
\centering
\includegraphics[width=0.35\textwidth]{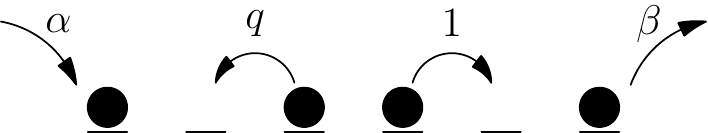}
\caption{The parameters of the PASEP.}
\noindent
\label{PASEP_parameters}
\end{wrapfigure}

In this work we consider a two-species PASEP with  ``heavy'' and ``light'' particles, which we call types 2 and 1 respectively. Both types of particles can swap places with an adjacent hole to the right and left with rates 1 and $q$, respectively. Furthermore, heavy particles can enter from the left and exit on the right of the lattice with respective rates $\alpha$ and $\beta$, and they can treat the light particles as holes and swap places with them to the right and left, also at rates 1 and $q$.  Since the light particles cannot enter or exit, their number stays fixed. In particular, when we fix the number of type 1 particles to be zero, we recover the original PASEP. The two-species process has been studied for some of its interesting thermodynamic properties, and a Matrix Ansatz solution and corresponding matrices in the work of Uchiyama in \cite{uchiyama} gave exact expressions for the steady state distribution of the system.

More formally, the two-species PASEP is a Markov chain, whose states are words of length $n$ in the letters $\{\D, \E, \A\}$, where D and A represent the particles of types 2 and 1, and E represents a hole.

Figure \ref{2_parameters} shows the parameters of the two-species process. More precisely, the transitions in the Markov chain are the following, with $X$ and $Y$ arbitrary words in $\{\D, \E, \A\}^{\ast}$.
\begin{displaymath} XAEY \overset{1}{\underset{q}{\rightleftharpoons}} XEAY \qquad XDEY \overset{1}{\underset{q}{\rightleftharpoons}} XEDY \qquad XDAY \overset{1}{\underset{q}{\rightleftharpoons}} XADY\end{displaymath}

\begin{displaymath}EX \overset{\alpha}{\rightharpoonup} DX \qquad \qquad XD \overset{\beta}{\rightharpoonup} XE\end{displaymath}
where by $X \overset{u}{\rightharpoonup} Y$ we mean that the transition from $X$ to $Y$ has probability $\frac{u}{n+1}$, $n$ being the length of $X$ (and also $Y$).

\begin{figure}[h]
\centering
\includegraphics[width=0.8\textwidth]{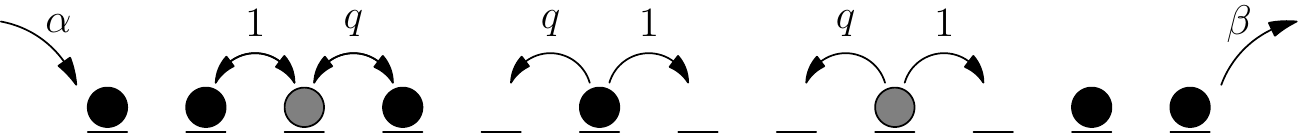}
\caption{The parameters of the two-species PASEP. The black particles are of type 2, and the grey ones are of type 1.}
\noindent
\label{2_parameters}
\end{figure}

Due to (2.9) of \cite{uchiyama}, a Matrix Ansatz solution holds for the two-species PASEP. For the theorem below, we use the following notation for a state: locations are labeled 1 through $n$, and $\tau_i \in \{0,1)$ and $\sigma_i \in \{0,1)$ are the particle indicators for particles 2 and 1 respectively, with the pair $(\tau,\sigma)$ representing a state. For example, the state $\A\E\D\D\A$ is represented by $\tau=(0,0,1,1,0)$ and $\sigma=(1,0,0,0,1)$. We use the notation $\Prob((\tau,\sigma))$ or equivalently $\Prob(X)$ for $X$ a word in $\{\D, \E, \A\}^{\ast}$ to describe the steady state probability of state $X$.

\begin{thm}[Uchiyama, 2008]\label{ansatz}
Let $(\tau,\sigma)$ represent a state of the two-species PASEP of length $n$ with $r$ particles of type 1. Suppose there are matrices D,E, and A and vectors $\langle w|$ and $|v \rangle$ which satisfy the following conditions
\begin{displaymath}DE = D+E+qED \qquad DA = A + qAD \qquad AE = A + qEA\end{displaymath}
\begin{displaymath}\langle w| E = \frac{1}{\alpha} \langle w| \qquad D|v \rangle= \frac{1}{\beta} |v \rangle\end{displaymath}
then 
\begin{displaymath}\Prob((\tau,\sigma)) = \frac{1}{Z_{n,r}} \langle w| \prod_{i=1}^n \tau_i D + \sigma_i A + (1-\tau_i)(1-\sigma_i)E|v \rangle\end{displaymath}
where $Z_{n,r}$ is the coefficient of  $y^r$ in $\frac{\langle w| (D + yA + E)^n|v \rangle}{\langle w | A^r|v \rangle}$.
\end{thm}

This result generalizes a previous Matrix Ansatz solution for the regular PASEP of Derrida et. al. in \cite{derrida}. 

In his paper, Uchiyama provides matrices (that are neither positive or rational) that satisfy the conditions of Theorem \ref{ansatz}. From these, the matrix product yields steady state probabilities in the form of polynomials in $\alpha$, $\beta$, and $q$ with positive integer coefficients. Therefore one would hope for a combinatorial interpretation of these probabilities, with results akin to those of Corteel and Williams \cite{cw2011} for the original PASEP. Such results could yield explicit general formulas for both the desired probabilities and the partition function.  

The goal of this paper is to provide some combinatorial solutions to this two-species problem for some special cases. In Section \ref{section_q0} of this paper, we describe certain tableaux which we call multi-Catalan tableaux that give an interpretation for the steady state distributions of the two-species PASEP at $q=0$. In Section \ref{section_enumerative} we provide some enumerative results for the multi-Catalan tableaux. In Section \ref{section_mc} we describe a Markov process on the multi-Catalan tableaux that projects to the two-species PASEP at $q=0$, and which gives another proof of our main result in Section \ref{section_q0}. Finally, in Section \ref{section_q1} we define some more general multi-Catalan tableaux that we believe give an interpretation for the steady state distributions of the two-species PASEP at $q=1$. Note that our forthcoming paper with X.~G.~Viennot \cite{TAT} will give another combinatorial solution to the two-species PASEP for general $q$.

\textit{Acknowledgement.} I would like to thank Sylvie Corteel and Lauren Williams for suggesting this problem to me, and for their invaluable advice and support. I also gratefully acknowledge the hospitality of LIAFA where this research was done, and the Chateaubriand Fellowship awarded by the Embassy of France in the United States that supported this stay in Paris.

\section{Multi-Catalan tableaux}\label{section_q0}

\begin{defn} \label{rules}
A \emph{multi-Catalan tableau} of \emph{size} $n$ is a filling of a Young diagram of shape $(n,n-1,\ldots,1)$ with the symbols $\alpha,\beta$, and $x$ as follows:
\begin{enumerate}
\item Every box on the diagonal must contain an $\alpha,\beta$, or $x$.
\item A box that sees an $\alpha$ to its right and a $\beta$ below must contain an $\alpha$ or $\beta$.
\item A box that sees an $\alpha$ to its right and an $x$ below must contain a $\beta$.
\item A box that sees an $x$ to its right and a $\beta$ below must contain an $\alpha$.
\item Every other box must be empty.
\end{enumerate}
\end{defn}

In the definition above, when we refer to the symbol that a box ``sees'' to its right or below, we mean the first symbol encountered in the same row or column, respectively. For example, in the first tableau of Figure \ref{DEEAE_example}, $x$ is the first symbol that the $\beta$ in the top row sees below it. Finally, note that Rule 5 implies that all boxes in the same row and left of a $\beta$ must be empty, and also that all boxes in the same column and above an $\alpha$ must be empty.

\begin{defn}\label{weight}
The \emph{weight} $\wt(T)$ of a multi-Catalan tableau $T$ is the product of all the $\alpha$'s and $\beta$'s it contains.
\end{defn}

\begin{defn}
The \emph{type} $\type(T)$ of the tableau $T$ is the word in $\{\D, \E, \A\}^{\ast}$ that is read off from the diagonal from top to bottom, by assigning a D to $\alpha$, an E to $\beta$, and an A to $x$.
\end{defn}

\begin{defn} 
The \emph{weight} of a word $X$ in $\{\D, \E, \A\}^{\ast}$ is 
\begin{displaymath}\weight(X)=\sum_T \wt(T),
\end{displaymath}
where the sum is over all multi-Catalan tableaux $T$ such that $\type(T)=X$.
\end{defn}

\begin{thm}\label{main_thm}
Consider the two-species PASEP on a lattice of $n$ sites. Let $X$ be a state described by a word in $\{\D, \E, \A\}^n$ with $r$ $\A$'s. Let $Z^0_{n,r} = \sum_{X'} \weight(X')$ where the sum is over all words $X'$ of length $n$ with $r$ $\A$'s. Then the steady state probability of state $X$ is
\begin{displaymath}\Prob(X)=\frac{\weight(X)}{Z^0_{n,r}}.\end{displaymath}
\end{thm}

\begin{figure}[h]
\centering
\includegraphics[width=0.6\textwidth]{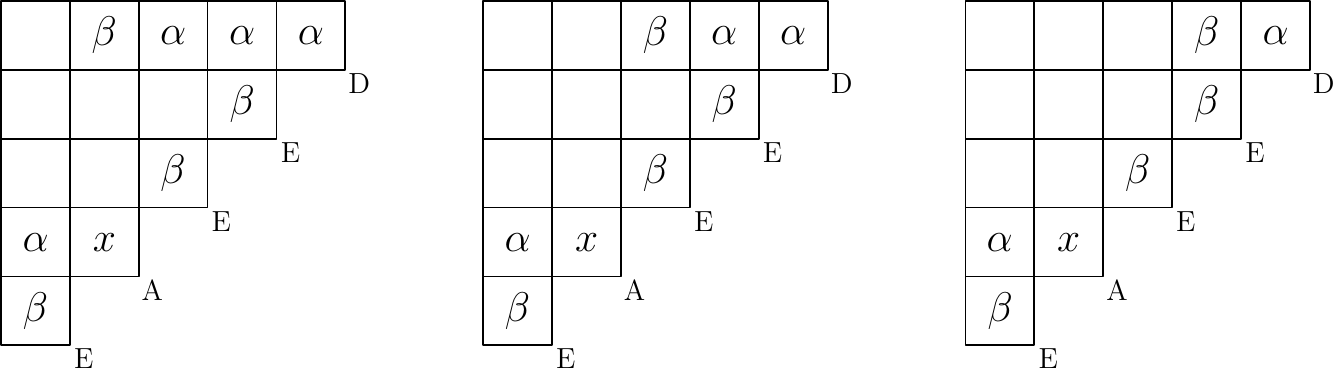}
\caption{These are all possible multi-Catalan tableaux of type $\D\E\E\A\E$. Their weights are $\alpha^4\beta^4, \alpha^3\beta^4$, and $\alpha^2\beta^4$ respectively.}
\noindent
\label{DEEAE_example}
\end{figure}

We show as an example all valid tableaux and their weights for the word DEEAE in Figure \ref{DEEAE_example}. Theorem \ref{main_thm} implies that
\begin{displaymath}
\Prob(\D\E\A\E\E) = \frac{1}{Z^0_{5,1}} \left( \alpha^4\beta^4+\alpha^3\beta^4+\alpha^2\beta^4\right).
\end{displaymath}

\begin{defn} 
A \emph{D-row} is a row whose right-most box contains an $\alpha$, and an \emph{A-row} is one whose right-most box contains an $x$. An \emph{E-column} is a column whose bottom-most box contains a $\beta$, and an \emph{A-column} is one whose bottom-most box contains an $x$. Then a \emph{DE box} is one that lies in a D-row and an E-column, (and correspondingly for DA, AE, and AA boxes).
\end{defn}

Note that we can ignore the rows with right-most box containing a $\beta$ or columns with bottom-most box containing an $\alpha$ because they are automatically required to be empty according to Definition \ref{rules}.

To connect back to the two-species TASEP, let a state of the TASEP be described by a word $X$ in $\{\D, \E, \A\}^n$. Then we fill a Young diagram of shape $(n,n-1,\ldots,1)$ as follows: from top to bottom, we fill the diagonal with symbols $\alpha,\beta$, and $x$ by reading the word $X$ from left to right, and placing an $\alpha$ for a D, a $\beta$ for an E, and an $x$ for an A. Then any valid filling of the rest of the diagram according to the rules (2)-(5) of Definition \ref{rules} will result in a multi-Catalan tableau of type $X$. 

\subsection{Condensed multi-Catalan tableaux}
We provide a condensed version of the characterization of the multi-Catalan tableaux, which offers a more natural proof of our results. We introduce the following definitions.

\begin{wrapfigure}[10]{r}{0.37\textwidth}
\centering
\includegraphics[width=0.2\textwidth]{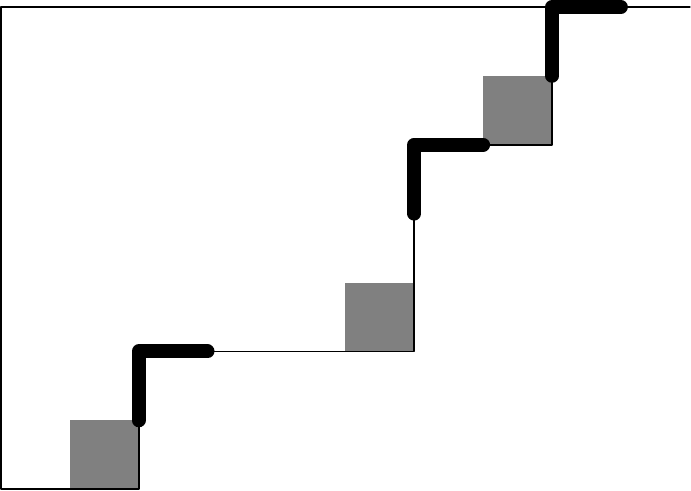}
\caption{The grey squares mark the corners of the diagram, and the darkened edges mark the inner corners.}
\noindent
\label{corner2}
\end{wrapfigure}

\begin{defn} An \emph{inner corner} of the multi-Catalan tableau is a consecutive pair of a west edge and a south edge on the boundary of the tableau. Figure \ref{corner2} shows some examples. A \emph{corner} is simply the box that is both the right-most box of some row and the bottom-most box of some column.
\end{defn}

\begin{defn}\label{condensed_def}
A \emph{condensed} multi-Catalan tableau $T$ of \textbf{size} $(n,k,r)$ is a Young diagram $Y=Y(T)$ with at least $r$ inner corners, that is justified to the northeast and contained in a rectangle of size $k+r \times n-k$. $Y$ is identified with the lattice path $L=L(T)$ that takes the steps south and west and follows the southeast border of $Y$. In addition, we have the following: 
\begin{itemize}
\item Each edge of $L$ is labelled with a D, E, or A such that exactly $r$ inner corners have both edges labeled with $\A$'s, and the remaining west edges have the label E, and the remaining south edges have the label D. 
\item An E-column is a column with an E labeling its bottom-most edge. (An A-column is defined correspondingly.)
\item A D-row is a row with a D labeling its right-most edge. (An A-row is defined correspondingly.)
\item A DE box is a box in a D-row and an E-column. (The DA, AE, AA boxes are defined correspondingly.)
\end{itemize}
Finally, we fill $T$ with $\alpha$'s and $\beta$'s according to the following rules:
\begin{enumerate}[i.]
\item A box in the same row and left of a $\beta$ must be empty.
\item A box in the same column and above of an $\alpha$ must be empty.
\item A DE box that is not forced to be empty must contain an $\alpha$ or a $\beta$.
\item A DA box that is not forced to be empty must contain a $\beta$.
\item An AE box that is not forced to be empty must contain an $\alpha$.
\end{enumerate}
\end{defn}

We identify the Young diagram $Y$ with a partition $\lambda = \lambda(T)$, which we also call the \emph{shape} of $Y$ and of $T$. Specifically, $\lambda=(\lambda_1,\ldots,\lambda_{k+r})$ where $\lambda_i$ is the number of boxes of $Y$ in row $i$ of the $k+r \times n-k$ rectangle.

\begin{defn}
The \emph{type} of the condensed version of the multi-Catalan tableau $T$ is the word in $\{\D, \E, \A\}^{\ast}$ that is read from the labels on the lattice path $L(T)$ from top to bottom, but with A counted only once for each pair of a west A-edge and a south A-edge in an inner corner of $T$.
\end{defn}

\begin{defn}
The \emph{weight} of the condensed tableau is the weight of the symbols inside it times the weight of the lattice path $L(T)$, which is obtained by giving each D edge weight $\alpha$ and each E edge weight $\beta$. In particular, for a tableau of size $(n,k,r)$, the weight of $L(T)$ is $\alpha^k \beta^{n-k-r}$. 
\end{defn}

In Figure \ref{condensed}, we demonstrate by example the conversion from a staircase multi-Catalan tableau to a condensed multi-Catalan tableau. Specifically, we remove the diagonal from the staircase along with the E-rows and D-columns, and then collapse all the DE, DA, AE, and AA boxes. Then we label the boundary edges of the tableau by labeling: a vertical edge with a D if it belongs to a D-row and with an A if it belongs to an A-row, and a horizontal edge with an E if it belongs to an E-column and with an A if it belongs to an A-column. It is easy to check that the types and weights of the two tableaux are equal.

Another way to obtain a condensed tableau from a word $X$  in $\{\D, \E, \A\}^{\ast}$ is to draw a lattice path $L=L(X)$ with steps south and west, by reading $X$ from left to right and drawing a D-labeled south edge for a D, an E-labeled west edge for an E, and an A-labeled pair of a west edge and a south edge for an A. $L$ is then identified with the Young diagram $Y$ whose southeast border it coincides with. (More precisely, $Y$ has shape $\lambda=(\lambda_1,\lambda_2,\ldots)$, where $\lambda_i$ is the number of $\E$'s and $\A$'s in $X$ following the $i$th instance of either D or A.) Any filling of $Y$ according to rules (i)-(v) of Definition \ref{condensed_def} results in a tableau of type $X$.

\begin{figure}[h]
\centering
\includegraphics[width=0.5\textwidth]{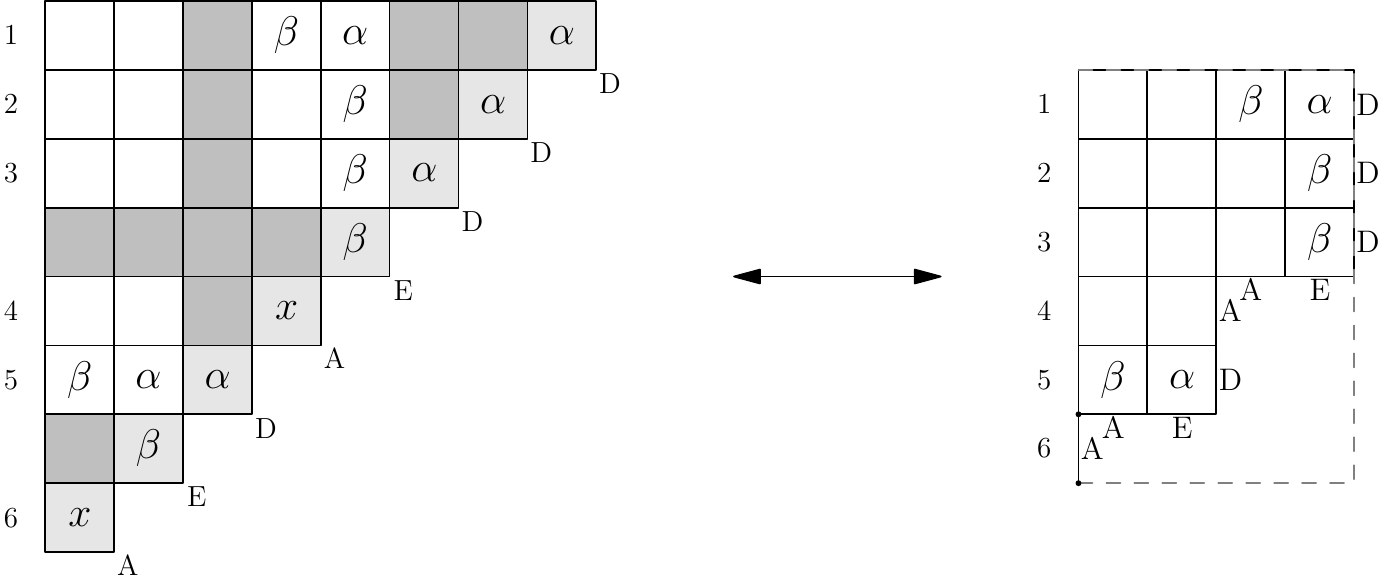}
\caption{The staircase multi-Catalan tableau and its corresponding condensed multi-Catalan tableau with size $(10,4,2)$ and shape $\lambda=(4,4,4,2,2,0)$ have type $\D\D\D\E\A\D\E\A$ and weight $\alpha^6\beta^6$. The collapsing together of the white boxes from the staircase tableau is the condensed tableau.}
\noindent
\label{condensed}
\end{figure}

Since the staircase version of the multi-Catalan tableaux is in simple bijection with the condensed version, we will call them both multi-Catalan tableaux, and refer to them interchangeably.

We now give a proof of Theorem \ref{main_thm}. To facilitate our proof, we provide a more flexible Matrix Ansatz that generalizes Theorem \ref{ansatz} with the same argument as in \cite[Theorem 5.2]{cw2011}. For a word $W \in \{\D, \E, \A\}^n$ with $r$ $\A$'s, we define unnormalized weights $f(W)$ (or equivalently $f((\tau,\sigma))$ for the corresponding notation $(\tau,\sigma)$ that represents $W$) which satisfy
\[
\Prob(W)=f(W)/Z_{n,r}
\]
where $Z_{n,r}=\sum_{W'}f(W')$ where the sum is over all words $W'$ of length $n$ and with $r$ $\A$'s.

\begin{thm}\label{ansatz2}
Let $\lambda$ be a constant. Let $\langle w|$ and $|v\rangle$ be row and column vectors with $\langle w| |v\rangle=1$. Let $D$, $E$, and $A$ be matrices such that for any words $X$ and $Y$ in $\{D, E, A\}^{\ast}$, the following conditions are satisfied:
\begin{enumerate}[(I)]
\item $\langle w| X(DE-qED)Y|v \rangle=\lambda \langle w| X(D+E)Y|v\rangle$,
\item $\langle w| X(DA-qAD)Y|v \rangle=\lambda \langle w| XAY|v\rangle$,
\item $\langle w| X(AE-qEA)Y|v \rangle=\lambda \langle w| XAY|v\rangle$,
\item $\beta \langle w| XD |v \rangle=\lambda \langle w| X|v\rangle$,
\item $\alpha \langle w| EY |v \rangle=\lambda \langle w| Y|v\rangle$.
\end{enumerate} 
Then for any state $(\tau,\sigma)$ of the two-species PASEP of length $n$, 
\[
f((\tau,\sigma))= \langle w| \prod_{i=1}^n \tau_i D + \sigma_i A + (1-\tau_i)(1-\sigma_i)E|v \rangle.
\] 
\end{thm}

The proof of Theorem \ref{ansatz2} follows exactly that of \cite[Theorem 5.2]{cw2011}. Note that the above implies that 
\[
Z_{n,r}=\frac{1}{\langle w| A^r |v \rangle} [y^r]\langle w| (D+yA+E)^n | v \rangle.
\]

\begin{proof} \emph{(Theorem \ref{main_thm})} The Matrix Ansatz of Theorem \ref{ansatz2} implies that the steady state probabilities for the two-species PASEP satisfy certain recurrences (that in turn determine all probabilities). The strategy of our proof is to show that the weight generating function for multi-Catalan tableaux of fixed type satisfies the same recurrences. Specifically, we use these recurrences with the constant $\lambda =\alpha\beta$, to show by induction that for $W$ a word in $\{\D, \E, \A\}^n$ with $r$ $\A$'s, 
\begin{equation}\label{prob}
f(W) = \weight(W).
\end{equation}

By Theorem \ref{ansatz2}, $f(W) =\langle w| W | v \rangle$, where we interpret $W$ as a product of the matrices $D$ , $E$ , and $A$  in the order those letters occur in the word $W \in \{\D, \E, \A\}^n$.
To start, if $n=1$, $W$ is either the word ``D'', ``E'', or ``A''. In each of those cases, there is a single tableau for $W$, and it is a trivial tableau. We obtain $\weight(\mbox{D})=\alpha$, $\weight(\mbox{E})=\beta$, and $\weight(\mbox{A})=1$. From Theorem \ref{ansatz}, it is clear that the base case indeed satisfies Equation \ref{prob}.



Now suppose that any word of length less than $n$ satisfies Equation \ref{prob}. Let $W$ be a word of length $n$ with $r$ $\A$'s. If $r=n$, then $\weight(W)=1$, and Equation \ref{prob} trivially holds. Thus assume $r<n$. Then at least one of the following must occur:
\begin{enumerate}[i.]
\item $W$ contains an instance of ``DE''.
\item $W$ contains an instance of ``DA''.
\item $W$ contains an instance of ``AE''.
\item $W$ begins with an E.
\item $W$ ends with a D.
\end{enumerate} 
We will express $\weight(W)$ in terms of the weight of some smaller word(s) of length $n-1$ based on the occurrence of one of the above. Let $T$ be any tableau of type $W$. Throughout the following, we let $X$ and $Y$ represent some arbitrary words in $\{\D, \E, \A\}^{\ast}$.

\emph{(i.) $W$ contains an instance of ``DE''.} Then $T$ has some DE corner box and we can write $W=X\D\E Y$. We call this DE corner box the \emph{chosen} corner. That chosen corner must contain either an $\alpha$ or a $\beta$, so we can decompose the possible fillings of $T$ into two cases as in Figure \ref{ansatz_0}. If the chosen corner contains an $\alpha$, then all the boxes above it are empty, and so its entire column has no effect on the rest of the tableau. Thus such $T$ can be mapped to a filling of a smaller tableau with that column removed, which would have type $X\D Y$. This map gives a bijection between tableaux of type $X\D\E  Y$ with an $\alpha$ in the chosen DE corner and tableaux of type $X\D Y$. The removed column with the $\alpha$ in its bottom-most box has total weight $\alpha\beta$.\footnote{In the total weight of a column, we include the weight of the bottom-most edge, which is a component of the southeast boundary of $T$. When the column removed is an E-column, the weight of the boundary component is $\beta$, so the total weight of the column with an $\alpha$ at the bottom is $\alpha\beta$. Similar reasoning is used in the other cases.} 

Similarly, if the chosen corner contains a $\beta$, then the boxes to its left must be empty, and so its entire row has no effect on the rest of the tableau. Hence such $T$ can be mapped to a smaller tableau with that row removed, which would have type $X\E Y$. This map gives a bijection between tableaux of type $X\D\E  Y$ with a $\beta$ in the chosen DE corner and tableaux of type $X\E Y$. The removed row with the $\beta$ in its right-most box also has total weight $\alpha\beta$. Consequently, we have the sum of the weights of the fillings:
\begin{displaymath} \weight(X\D\E Y) = \weight(X\D Y) \cdot \alpha\beta + \weight(X\E Y) \cdot \alpha\beta. \end{displaymath}
By the induction hypothesis, since the lengths of $X\D Y$ and $X\E Y$ are both $n-1$, we have $\weight(X\D Y) =  \langle w| X D Y |v \rangle$ and $\weight(X\E Y) = \langle w| X E Y |v \rangle$. By Theorem \ref{ansatz2} with $\lambda=\alpha\beta$, 
\[
(\alpha\beta) \langle w| X (D+E) Y | v \rangle = \langle w| X DE Y | v \rangle  = \langle w| W | v \rangle.
\]
It follows that $W$ satisfies Equation \ref{prob}.

\begin{figure}
\centering
\includegraphics[width=0.7\textwidth]{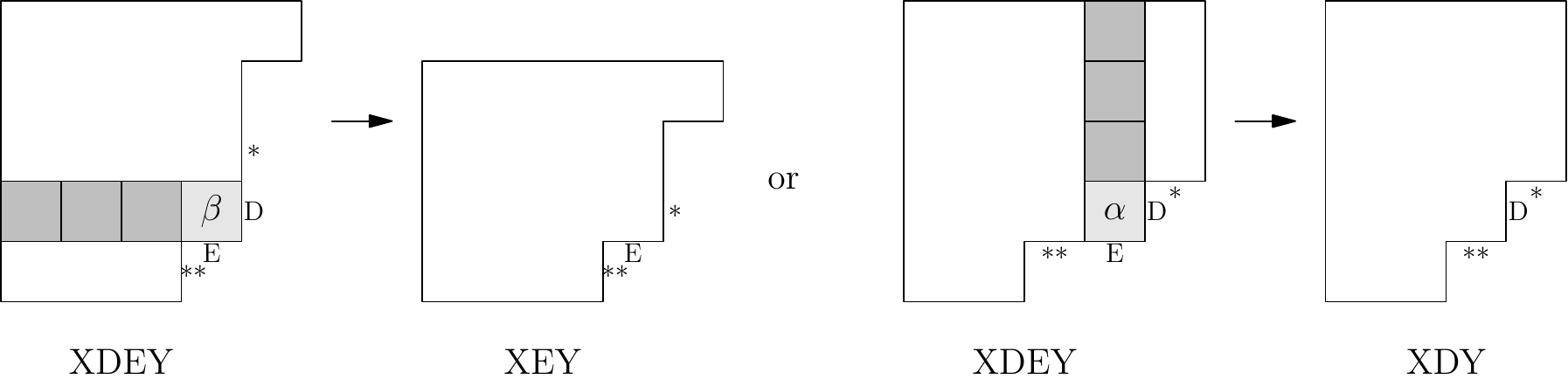}
\caption{The chosen corner is a DE box.}
\noindent
\label{ansatz_0}
\end{figure}

\begin{figure}
\centering
\includegraphics[width=0.7\textwidth]{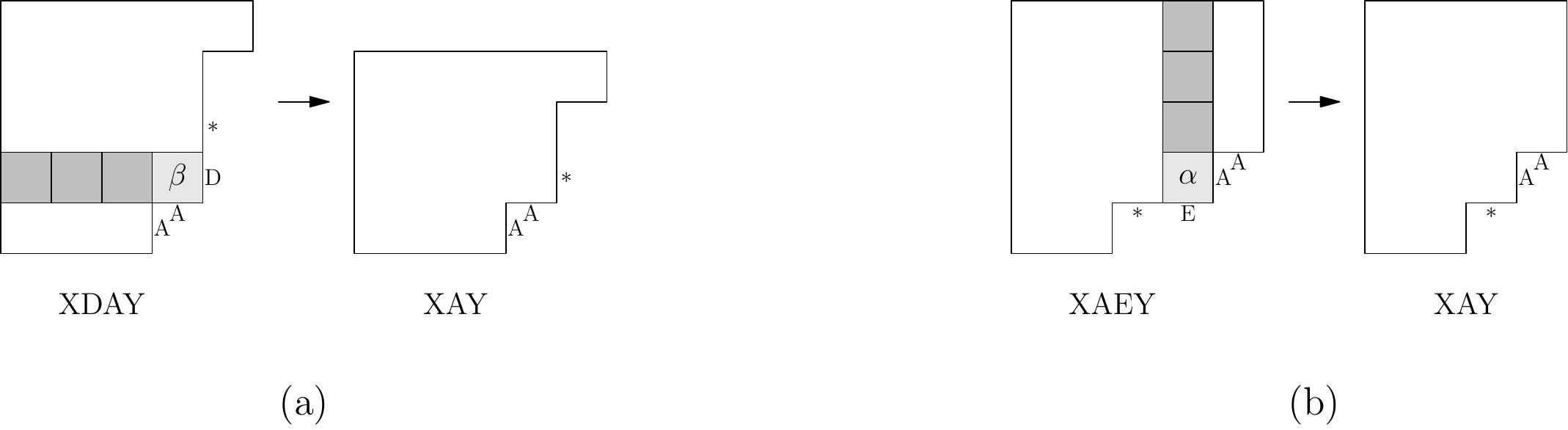}
\caption{The chosen corner is (a) a DA box and (b) an AE box.}
\noindent
\label{ansatz_1}
\end{figure}

\emph{(ii.) $W$ contains an instance of ``DA''.} Then $T$ has some DA corner box and we can write $W=X\D\A  Y$. That box necessarily contains a $\beta$, so the boxes to its left are empty and the entire row has no effect on the rest of the tableau. Thus we can map $T$ to a filling of a smaller tableau with that row removed, which would have type $X\A Y$, as in Figure \ref{ansatz_1} (a). This map gives a bijection between tableaux of type $X\D\A  Y$ and tableaux of type $X\A Y$. The removed D-row with the $\beta$ in its right-most box has total weight $\alpha\beta$. Thus we obtain the sum of the weights of the fillings:
\begin{displaymath} \weight(X\D\A Y) = \weight(X\A Y) \cdot \alpha\beta.\end{displaymath}
Similar reasoning to the DE case completes the argument.

\emph{(iii.) $W$ contains an instance of ``AE''.} Then $T$ has some AE corner box and we can write $W=X\A\E  Y$. That box necessarily contains an $\alpha$, so the boxes above it are empty and the entire column has no effect on the rest of the tableau. Thus we can map $T$ to a filling of a smaller tableau with that column removed, which would have type $X\A Y$, as in Figure \ref{ansatz_1} (b). This map is a bijection between tableaux of type $X\A\E  Y$ and tableaux of type $X\A Y$. The removed E-column with the $\alpha$ in its bottom-most box has total weight $\alpha\beta$. Thus we obtain the sum of the weights of the fillings:
\begin{displaymath} \weight(X\A\E Y) = \weight(X\A Y) \cdot \alpha\beta.\end{displaymath}
Similar reasoning to the DE case completes the argument.
 
\emph{(iv.) $W$ begins with an E.} Then $T$ has one or more empty E columns on the east end of $Y(T)$. We can write $W=\E X$. The presence of this empty column does not affect the filling of $T$, so its removal would result in a map to a smaller tableau of type $X$ with the same filling. This map is a bijection between tableaux of type $\E X$ and tableaux of type $X$. The removed E-column has total weight $\beta$. Thus we obtain the sum of the weights of the fillings:
\begin{displaymath} \weight(\E X) =\weight(X)\cdot \beta.\end{displaymath}
Similar reasoning to the DE case completes the argument.

\emph{(v.) $W$ ends with a D.} Then $T$ has one or more empty D rows on the south end of $Y(T)$. We can write $W=X\D$. The presence of this empty row does not affect the filling of $T$, so its removal would result in a map to a smaller tableau of type $X$ with the same filling. This map is a bijection between tableaux of type $X\D$ and tableaux of type $X$.  Thus we obtain the sum of the weights of the fillings:
\begin{displaymath}\weight(X\D) = \weight(X) \cdot \alpha.\end{displaymath}
Similar reasoning to the DE case completes the argument.

From the above cases, we obtain that any word $W$ of length $n$ satisfies Equation \ref{prob}, which is the desired result.
\end{proof}

\section{Enumeration of multi-Catalan tableaux}\label{section_enumerative}

Building on some enumerative results in \cite{mandelshtam} for regular Catalan tableaux that correspond to the regular one-species PASEP, we can deduce the following for the multi-Catalan tableaux.


\begin{thm}\label{thm_total}
 The number of multi-Catalan tableaux corresponding to a two-species PASEP at $q=0$ of size $n$ and with $r$ particles of type 1 is
 \begin{displaymath}Z^0_{n,r}(\alpha=\beta=1) = \frac{2(r+1)}{n+r+2}{2n+1 \choose n-r}.\end{displaymath}
\end{thm}

\begin{proof}
We make two observations about the structure of the tableaux. First, any box that lies in an A-row or column is either empty or automatically determined by the rules (iv)-(v) from Definition \ref{condensed_def}. Second, any box that lies \emph{left} of an A-column or \emph{above} an A-row must be empty. In particular, note that any box that lies in an A-column is either empty if ther$\E$'s already a $\beta$ to the right in the same row, or is forced to contain a $\beta$ otherwise. In both of these cases, any box to the left of that A-column must be empty. Similarly, any box that lies in an A-row is either empty if ther$\E$'s already an $\alpha$ below in the same column, or must contain an $\alpha$ otherwise. In both of these cases, any box above that A-row must be empty. Figure \ref{fig_multicatalan_example} shows an example of this structure.

Consequently, the filling of the multi-Catalan tableau can be recreated from the fillings of just the DE boxes that do not lie north or west of any A-rows or columns. Thus to enumerate these fillings, we can remove all the boxes that lie in the $r$ A-rows and columns along with all the boxes respectively north and west of these rows and columns. We are left with a disjointed set of $r+1$ smaller tableaux, each of which is a multi-Catalan tableaus whose type has zero $\A$'s. The sum of the sizes of these $r+1$ tableaux is $n-r$. 

Multi-Catalan tableaux whose type has zero $\A$'s are the same as the Catalan tableaux from \cite{mandelshtam}, which are a specialization of a number of objects that are well known in the literature. In particular, the number of such tableaux of size $n$ is the Catalan number $C_{n+1} = \frac{1}{n+1}{2n+2 \choose n+1}$. We obtain the equation in the theorem as the appropriately chosen coefficient of the convolution of $r+1$ Catalan numbers.
\end{proof}

\begin{figure}
\centering
\includegraphics[width=0.3\textwidth]{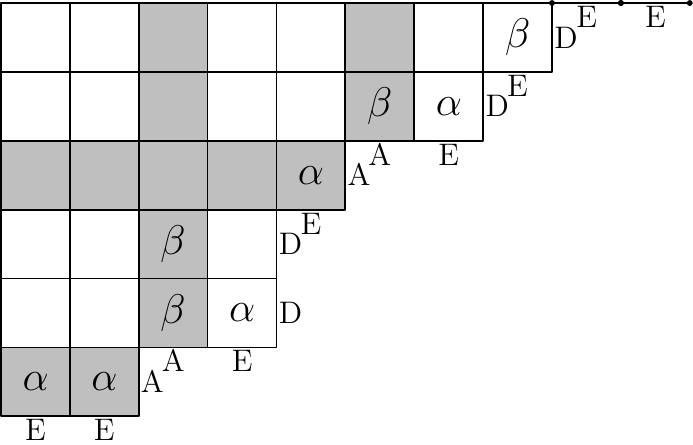}
\caption{The grey boxes indicate the boxes that belong to an A-row or A-column. Observe that any box above an A-row or left of an A-column is forced to be empty.}
\noindent
\label{fig_multicatalan_example}
\end{figure}



\begin{thm}
The number of two-species multi-Catalan tableaux corresponding to a two-species PASEP at $q=0$ of size $n$ and with $r$ particles of type 1 and $k$ particles of type 2 is 
\begin{displaymath}\sum\limits_
{\substack{k_1\leq n_1,\ldots,k_{r+1}\leq n_{r+1},\\
n_1+\cdots+n_{r+1}=n-r\\
k_1+\cdots+k_{r+1}=k}} \prod_{i=1}^{r+1} \mathcal{N}_{n_i+1,k_i+1}\end{displaymath}
where $\mathcal{N}_{n,k}=\frac{1}{n}{n \choose k}{n \choose k-1}$ is the $n,k-$ Narayana number.
\end{thm}

\begin{proof}
We refine the proof of Theorem \ref{thm_total} by keeping track of the number of D rows. More precisely, after removing the A rows and columns and the boxes that lie respectively above and west of the A rows and columns, we are left with a disjointed list of $r+1$ smaller tableaux, the sum of whose sizes is $n-r$. We let these smaller tableaux (starting from top to bottom) have sizes $n_1,\ldots,n_{r+1}$ with $n_1+\cdots+n_{r+1}=n-r$. Furthermore, if we wish to have a total of $k$ D rows, we let the smaller tableaux have, respectively, $k_1,\ldots,k_{r+1}$ D rows with $k_1+\cdots+k_{r+1}=k$.

It is known from the literature that the number of multi-Catalan tableaux of size $n$ whose type has zero $\A$'s and $k$ $\D$'s is $\mathcal{N}_{n_i+1,k_i+1}$, from which the theorem follows. 
\end{proof}

For the next theorem, we make some more precise definitions to describe the structure of the multi-Catalan tableaux.
\begin{defn}\label{Lambdas}
We represent a word $X$ in $\{\D, \E, \A\}^n$ with exactly $r$ $\A$'s by a list of $r+1$ words in $\{\D, \E\}^{\ast}$, where each word of the list is the longest possible continuous sub-word of $X$ that does not contain an $\A$. We call this list of D-E sub-words $(X_1, \ldots, X_{r+1})$. We then represent that list by a list of partitions $\Lambda = (\Lambda_1, \ldots, \Lambda_{r+1})$, where the partition $\Lambda_{i}=\lambda(X_i)$ is the shape obtained from applying the definition of the partition $\lambda$ to the $i$th D-E word.
\end{defn}

As an example for the above, the tableau in Figure \ref{fig_multicatalan_example} has type $\E\E\D\E\D\E\A\E\D\D\E\A\E\E$, which can be rewritten as a list of three D-E words $(\E\E\D\E\D\E, \E\D\D\E, \E\E)$. Then the list of partitions is $\Lambda=((2,1), (1,1), (\emptyset))$.

For our final result in this section, we define the matrix $A_{\lambda}^{\alpha,\beta}=(A_{ij})_{1\leq i,j\leq k}$, where $\lambda$ is some partition $(\lambda_1,\ldots,\lambda_k)$, and
\begin{displaymath}
{\textstyle A_{ij} = \beta^{j-i}\alpha^{\lambda_i-\lambda_{j+1}} \left({\lambda_{j+1} \choose j-i}+\beta {\lambda_{j+1} \choose j-i+1}\right) +\beta^{j-i}\alpha^{\lambda_{i}-\lambda_{j}} \sum_{\ell=0}^{\lambda_j-\lambda_{j+1}-1} \alpha^{\ell} \left({\lambda_j-\ell-1 \choose j-i-1}+\beta {\lambda_j-\ell-1 \choose j-i} \right)}.
\end{displaymath}
From \cite{mandelshtam}, $\weight(X)=\det A_{\lambda(X)}^{\alpha,\beta}$ for $X$ a word in $\{\D, \E\}^{\ast}$ corresponding to state of the two-species PASEP with zero type 1 particles. 

\begin{thm}
Consider the two-species PASEP of size $n$ at $q=0$, and a state $X$ with exactly $r$ type 1 particles. Let $\Lambda = (\Lambda_1,\ldots,\Lambda_{r+1})$ be the list of partitions that corresponds to $X$ according to Definition \ref{Lambdas}. Let $\Lambda_i$ have $k_i$ rows and $m_i$ columns. Then:

\begin{displaymath}\pi(X) = \alpha^{n-m_1} \beta^{n-k_{r+1}} \det A_{\Lambda_1}(\alpha,1) \det A_{\Lambda_{r+1}}(1,\beta) \prod_{i=2}^r \det A_{\Lambda_i}(1,1)\end{displaymath}
is the unnormalized steady state probability of state $X$.
\end{thm}

\section{A Markov chain on the multi-Catalan tableaux that projects to the two-species PASEP at $q=0$}\label{section_mc}

In this section we construct a Markov chain on the multi-Catalan tableaux that provides a second proof of Theorem \ref{main_thm} and generalizes the construction of Corteel and Williams from \cite{cw_mc}. We start by defining projection for Markov chains, from \cite[Definition 3.20]{cw_mc}. 

\begin{defn}
Let $M$ and $N$ be Markov chains on finite sets $X$ and $Y$, and let $F$ be a surjective map from $X$ to $Y$. We say that $M$ \emph{projects} to $N$ if the following properties hold:
\begin{itemize}
\item If $x_1, x_2 \in X$ with $Prob_M(x_1 \rightarrow x_2) > 0$, then $Prob_M(x_1 \rightarrow x_2) = Prob_N(F(x_1) \rightarrow F(x_2))$.
\item If $y_1$ and $y_2$ are in $Y$ and $Prob_N(y_1 \rightarrow y_2)>0$, then for each $x_1 \in X$ such that $F(x_1) = y_1$, there is a unique $x_2 \in X$ such that $F(x_2) = y_2$ and $Prob_M (x_1 \rightarrow x_2) > 0$; moreover, $Prob_M (x_1 \rightarrow x_2) = Prob_N (y_1 \rightarrow y_2)$.
\end{itemize}
\end{defn}

This means that: if $M$ projects to $N$ via the map $F$, then the steady state probability that $N$ is in state $y$ is equal to the sum of the steady state probabilities over all the states $x \in \{z \in X|F(z)=y\}$. In our case, $N$ is the two-species PASEP at $q=0$, and $M$ is the Markov chain on the multi-Catalan tableaux which we describe below. Corteel and Williams defined a Markov chain on permutation tableaux (in bijection with alternative tableaux) that projects to the PASEP. In the two-species PASEP at $q=0$, we have an analogous result using similar transitions.

\begin{wrapfigure}[11]{r}{0.4\textwidth}
\centering
\includegraphics[width=0.4\textwidth]{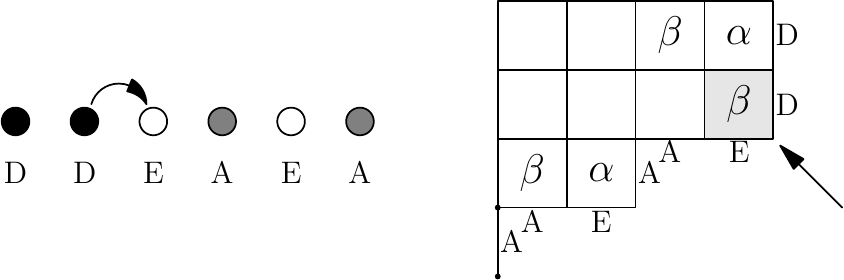}
\caption{The tableau corner associated to the transition $DDEAEA \rightarrow DEDAEA$.}
\noindent
\label{hopping_example}
\end{wrapfigure}

\begin{defn}
Define a \emph{corner} of the tableau to be a DE, DA, or AE box that is both the right-most box of some row and the bottom-most box of some column. Define a \emph{right leg} to be the set of E-edges of the lattice path $L(T)$ that lie on the north boundary of the $k+r \times n-k$ rectangle. In other words, $T$ has a right leg when $\type(T)$ begins with E. Analogously, define a \emph{left leg} to be the set of D-edges of the lattice path $L(T)$ that lie on the west boundary of the $k+r \times n-k$ rectangle. In other words, $T$ has a left leg when $\type(T)$ ends with a D. We call the \emph{transition points} the union of the set of corners along with the left leg and right leg (if those are present).
\end{defn}

We describe the process by examining the possible transitions out of some multi-Catalan tableau $T$ which corresponds to the two-species PASEP state $X$ for which $\type(T)=X$. Every transition is associated to some chosen transition point (namely, either a chosen corner or a right leg or a left leg).  In particular, the right leg corresponds to a transition $\E X' \rightarrow\D X'$ for $X=\E X'$, and the left leg corresponds to a transition $X'\D \rightarrow X'\E$ for $X=X'\D$. On the other hand, each corner of $T$ corresponds to a transition from $X$ on the PASEP that does not involve particles entering or exiting at the boundary. Specifically, a transition on state $X$ at PASEP location $i$ corresponds to a transition at the corner box in $T$ that has its east edge precisely the $i$th edge of the lattice path $L(T)$ (from top to bottom). Figure \ref{hopping_example} shows a transition location in a two-species PASEP word $X$ along with the corresponding corner of some multi-Catalan tableau of type $X$. Moreover, a DE corner of $T$ corresponds to a transition $\D\E \rightarrow \E\D$, a DA corner corresponds to a transition $\D\A \rightarrow \A\D$, and an AE corner corresponds to a transition $\A\E \rightarrow \E\A$ out of state $X$. (We note here that AA boxes are not included in the set of corners since there is no PASEP transition they correspond to.)

To obtain a transition at a chosen corner, we first strip off the labels on the boundary, then perform certain column or row removal and re-insertion, and finally reapply new labels. We describe the column/row procedure below for the two possible cases for the Greek symbol that corner box could contain.

\begin{figure}[!ht]
\centering
\includegraphics[width=\textwidth]{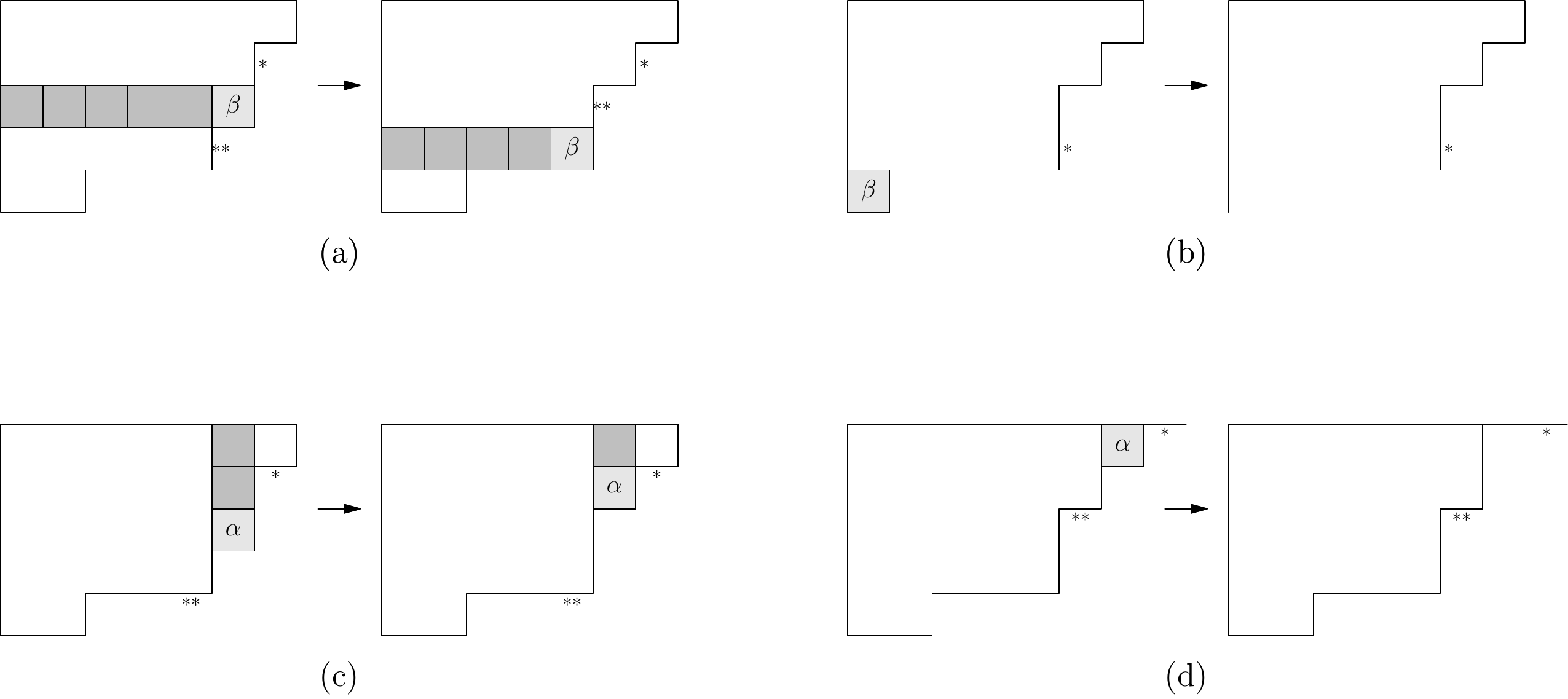}
\caption{The row removal and re-insertion procedure for (a) and (b) the chosen corner containing a $\beta$, and (c) and (d) the chosen corner containing an $\alpha$.}
\noindent
\label{mc_transitions}
\end{figure}
 
\noindent \textbf{The corner contains a $\boldsymbol{\beta}$.} Remove the row containing the corner (which is a horizontal stack of empty boxes with a $\beta$ on the right), cut off one of the empty boxes, and insert the row (with the $\beta$ still at the right of it) in the bottom-most location possible so that the resulting shape is still a Young shape. Figure \ref{mc_transitions} (a) shows an example. If the row originally had a single box, cutting off a box means it becomes an empty row, and so it should be placed at the south end of the shape with the rest of the empty rows. Figure \ref{mc_transitions} (b) shows an example of this case.

\noindent \textbf{The corner contains an $\boldsymbol{\alpha}$.} Remove the column containing the corner (which is a stack of empty boxes above an $\alpha$), cut off one of the empty boxes, and insert the column (with the $\alpha$ still at the bottom of it) in the right-most location possible so that the resulting shape is still a Young shape. Figure \ref{mc_transitions} (c) shows an example. If the column originally had a single box, cutting off a box means it becomes an empty column, and so it should be placed at the east end of the shape with the rest of the empty columns. Figure \ref{mc_transitions} (d) shows an example of this case.

Now we put the labels back on the edges of the boundary after exchanging the relevant two letters in the labelling word. For example, if the original state was $X\D\E Y$ for some words $X$ and $Y$, and a type 2 particle hopped to get the state $X\E\D Y$, then the labels on the boundary change from $X\D\E  Y$ to $X\E\D Y$. 

The following lemmas verify that the above actions are well-defined.

\begin{lemma}\label{beta_lemma}
Let $X$ be a word in $\{\D, \E, \A\}^{\ast}$, and let $T$ be a multi-Catalan tableau with $\type(T)=X$. A transition as defined above \textbf{at a corner that contains a $\boldsymbol{\beta}$} 
results in a valid multi-Catalan tableau.
\end{lemma}

\begin{proof}
Let $Y$ be the Young diagram and $\lambda=(\lambda_1,\ldots,\lambda_{r+k})$ be the partition associated to $T$ with $r$ and $k$ the number of $\A$'s and $\D$'s respectively in $\type(T)$. Let $L$ be the lattice path associated to $T$. The edges of $L$ are labeled with $\D$'s, $\E$'s, and $\A$'s according to the labelling word $X$. Let $X'$ and $X''$ denote arbitrary words in $\{\D, \E, \A\}^{\ast}$.

\noindent \emph{Transition at a DE corner.}
Suppose the chosen corner is a DE corner of tableau $T$ in row $i$ of length $\lambda_i$, so we write $X=X'\D\E  X''$. The fact that this box is a corner implies that for any $j>i$, $\lambda_{j}<\lambda_i$. Thus after removing row $i$ of $T$ and reinserting a row of length $\lambda_i-1$ into the lowest position possible, we obtain a tableau $T'$ of shape $(\lambda_1,\ldots,\lambda_{i-1},\lambda_i - 1, \lambda_{i+1},\ldots,\lambda_{r+k})$, which is in fact the shape $\lambda$ with the single box removed in row $i$. In other words, the pair of south and west edges of $L(T)$ that correspond to the DE corner of row $i$ change places, along with their labels of D and E. The lattice path $L(T')$ associated to $T'$ has labelling word $X'\E\D X''$. 

First, if $\lambda_i=1$, the transition is completed by simply removing row $i$ and replacing it with a single south edge on the west boundary of $T$, and replacing the labels with the labelling word $X'\E\D X''$ to obtain the new tableau $T'$. In this case, the weight of the boundary of $T'$ is the same as for $T$, but the filling of $T'$ has lost one $\beta$, so $wt(T') = \frac{1}{\beta}\wt(T)$. We observe that if $\lambda_i=1$, then $X$ must necessarily have the form $X'\D\E\D^k$.

It remains to check that inserting a row of length $\lambda_i-1$ for $\lambda_i>1$ with a $\beta$ in its right-most box into the lowest position possible results in a valid filling of $T'$. Suppose the row was inserted right above row $j$. Then necessarily $\lambda_j<\lambda_i-1$, and so the right-most box of the inserted row does not lie above any other box. Thus it is valid to place a $\beta$ in it as long as this $\beta$ is not in an A-row.

Here we make the following observation: each $\A$ in $X$ corresponds to the consecutive pair of a west edge followed by a south edge in the path $L(T)$. Thus if some row $s$ with length $\lambda_s$ is an A-row, then necessarily $\lambda_{s-1} \geq \lambda_s+1$. In other words, if $\lambda_{s-1} = \lambda_s + 1$, then the row $s-1$ necessarily ends with a corner box that is in an A-column.

Therefore, the only way for the newly inserted row (which has length $\lambda_i-1$) to end up being an A-row is if the row above it has length $\lambda_i$ and ends with a corner box that is in an A-column. This implies that the chosen corner of $T$ was in that same A-column. That is not possible since we started with the condition that that corner box is a DE box. Therefore, we can safely place a $\beta$ in the right-most box of the newly inserted row, and so this new row does not interfere with the rest of the filling of the tableau. We have thus a valid filling of a multi-Catalan tableau $T'$ of type $X'\E\D X''$. 

From the above, if $\lambda_i>1$, neither the weight of the filling or the weight of the boundary of the tableau changed after the transition and so $wt(T') =\wt(T)$.

\noindent \emph{Transition at a DA corner.}
Suppose the chosen corner is a DA corner of tableau $T$ in row $i$ of length $\lambda_i$, so we write $X=X'\D\A  X''$. Recall that a west edge in the lattice path $L(T)$ that is labelled by A  is necessarily followed by a south edge that is also labelled by A. Thus $\lambda_{i+1}=\lambda_i-1$. 

After removing row $i$ of $T$ and reinserting a row of length $\lambda_i-1$ into the lowest position possible, we obtain a tableau $T'$ of shape $(\lambda_1,\ldots,\lambda_{i-1},\lambda_i - 1,\lambda_i-1, \lambda_{i+2},\ldots,\lambda_{r+k})$, which is in fact the shape $\lambda$ with the single box removed in row $i$. But now, if apply the labelling word $X'\A\D X''$ to $T'$, we are in fact changing three of the edges in $L(T)$. Specifically, the pair of a south edge and a west edge of $L(T)$ that correspond to the DA corner of row $i$ change places, and become a consecutive pair of a west edge and a south edge which we can both label with A in the new lattice path $L(T')$, as in Figure \ref{DA}. Since we now have both rows $i$ and $i+1$ of $Y'$ of length $\lambda_i-1$, the south edge at the boundary of row $i+1$ of $T'$ must necessarily have a D label. Consequently we have described how we obtained the tableau $T'$ and its associated lattice path $L(T')$ whose labelling word is $X'\A\D X''$. 

Now, exactly as in the DE case, if $\lambda_i=1$, the transition is completed by simply removing row $i$ and replacing it with a single south edge on the west boundary of $T$, and replacing the labels with the labelling word $X'\A\D X''$ to obtain the new tableau $T'$. In this case, the weight of the boundary of $T'$ is the same as for $T$, but the filling of $T'$ has lost one $\beta$, so $wt(T') = \frac{1}{\beta}\wt(T)$. We observe that if $\lambda_i=1$, then $X$ must necessarily have the form $X'\D\A\D^k$.

\begin{figure}[h]
\centering
\includegraphics[width=\textwidth]{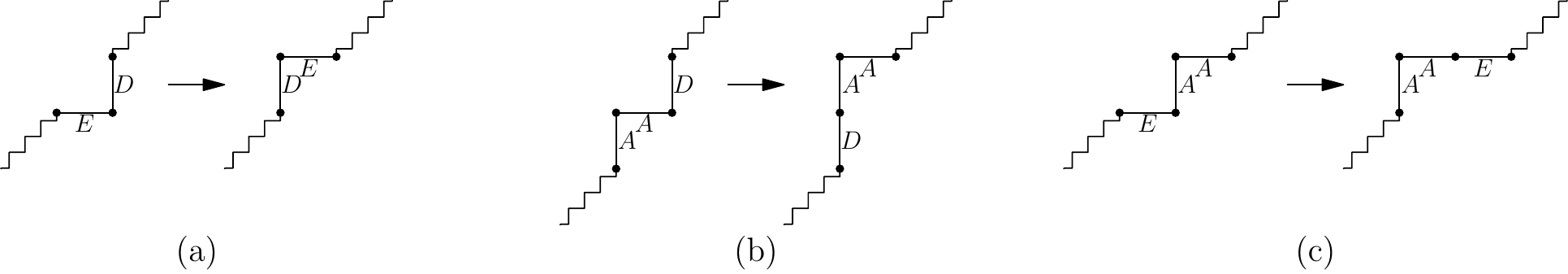}
\caption{The transition on a multi-Catalan tableau that corresponds to the PASEP transition (a) DE$\rightarrow$ED, (b) DA$\rightarrow$AD and (c) AE$\rightarrow$EA.}
\noindent
\label{DA}
\end{figure}

It remains to check that inserting a row of length $\lambda_i-1$ for $\lambda_i>1$ with a $\beta$ in its right-most box results in a valid filling of $T'$. Suppose the row was inserted right above row $j$. Then necessarily $\lambda_j<\lambda_i-1$, and so the right-most box of the inserted row does not lie above any other box. Thus it is valid to place a $\beta$ in it as long as this $\beta$ is not in an A-row. However, now we use the fact that $T'$ has both rows $i$ and $i+1$ of length $\lambda_i-1$. If the newly inserted row of length $\lambda_i-1$ were to end up an A-row, that would imply the row directly above it has length one box more, which contradicts both rows $i$ and $i+1$ having the same length $\lambda_i-1$. Therefore, the newly inserted row is necessarily a D-row. Thus we can safely place a $\beta$ in its right-most box, and so this new row does not interfere with the rest of the filling of the tableau. We have thus a valid filling of a multi-Catalan tableau $T'$ of type $X'\A\D X''$. 

From the above, if $\lambda_i>1$, neither the weight of the filling or the weight of the boundary of the tableau changed after the transition and so $wt(T') =\wt(T)$.
\end{proof}

\begin{lemma}
Let $X$ be a word in $\{\D, \E, \A\}^{\ast}$, and let $T$ be a multi-Catalan tableau with $\type(T)=X$. A transition as defined above \textbf{at a corner that contains an $\boldsymbol{\alpha}$} results in a valid multi-Catalan tableau.
\end{lemma}

\begin{proof}
By the symmetry of the rules for the multi-Catalan tableaux, the proof is exactly the same as the one for Lemma \ref{beta_lemma}, except if we take the transpose of the tableau and exchange the roles of $\alpha$ and $\beta$. As before, see Figure \ref{DA} for the transition $L(T) \rightarrow L(T')$. It will be useful further on that if the transition from $T$ to $T'$ occurs at a corner that belongs to a column of length $\mu_i>1$, then as before, $\wt(T')=\wt(T)$. Otherwise, if $\mu_i=1$, then $\wt(T') = \frac{1}{\alpha}\wt(T)$. In that case, if the transition occurred at a DE corner, $\type(T)$ necessarily has the form E$^{\ell}\D\E  X$, and if the transition occurred at an AE corner, $\type(T)$ necessarily has the form E$^{\ell}\A\E  X$ for some word $X$ in $\{\D, \E, \A\}^{\ast}$.
\end{proof}

{\bf Transitions at the boundary.} For an arbitrary PASEP word $X$ in $\{\D, \E, \A\}^{\ast}$, we describe the transition that corresponds to the PASEP transition $\E X \rightarrow \D X$ from  a tableau $T$ of type $\E X$ to a tableau $T'$ of type $\D X$. $T$ must necessarily have at least one empty E-column on its right, so after stripping off the labels of the tableau, we remove the right-most empty column and instead insert a row with a $\beta$ in its right-most box, of maximal possible length such that the semi-perimeter stays fixed, but at the lowest position possible for that length. (If it is not possible to insert row of nonzero length, we simply insert a row of length zero and do not add a $\beta$ to the filling.) Finally, we apply the labeling word $\D X$ to the edges of $L(T')$. Figure \ref{EX} (a) shows an example.

\begin{figure}[h]
\centering
\includegraphics[width=0.7\textwidth]{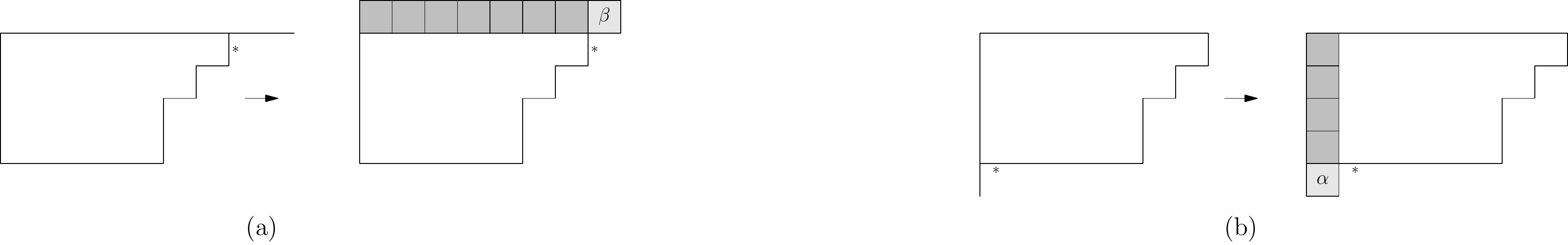}
\caption{The transition on a multi-Catalan tableau that corresponds to the PASEP transition (a) $\E X\rightarrow\D X$ and (b) $X\D \rightarrow X\D$.}
\noindent
\label{EX}
\end{figure}

Let such a transition occur with rate $\alpha$.

If a new row of nonzero length was inserted into $T$, then effectively one E-labeled boundary edge of $T$ was replaced with a D-labeled boundary edge for $T'$, plus the filling of $T'$ gained one $\beta$. Then $\wt(T') = \alpha \wt(T)$. Otherwise, if the new row had length zero, we simply replaced one E-labeled boundary edge with a D-labeled boundary edge, so $\wt(T') = \frac{\alpha}{\beta} \wt(T)$. In this last case, $T$ necessarily has type $\E\D^k$. 

The transition $X\D \rightarrow X\E$ is symmetric to the above, except that we need to take the transpose of the tableau and exchange the roles of $\alpha$ and $\beta$. For a transition from $T$ of type $X\D$ to $T'$ of type $X\E$, $T$ must necessarily have at least one empty D-row at the bottom of it. After stripping off the labels of the tableau, we remove the bottom-most empty row and instead insert a column with an $\alpha$ in its bottom-most box, of maximal possible length such that the semi-perimeter stays fixed, at the right-most position possible for that length. (If it is not possible to insert a column of nonzero length, we simply insert a column of length zero and do not add an $\alpha$ to the filling.) Finally, we apply the labeling word $X\E$ to the edges of $L(T')$. Figure \ref{EX} (b) shows an example.

Let such a transition occur with rate $\beta$.  

Similarly to the above, if the new column added has nonzero length, we obtain that $\wt(T')=\beta \wt(T)$. Otherwise, if the new column has length zero, $\wt(T') = \frac{\beta}{\alpha} \wt(T)$, and in this last case $T$ necessarily has type $\E^{\ell}\D$.

The following lemmas prove that these boundary transitions are well-defined.

\begin{lemma}\label{EX_lemma}
Let $X$ be a word in $\{\D, \E, \A\}^{\ast}$, and let $T$ be a multi-Catalan tableau with $\type(T)=\E X$. A transition on the boundary of $T$ as defined above, corresponding to the PASEP transition $\E X \rightarrow\D X$ results in a valid multi-Catalan tableau.
\end{lemma}

\begin{proof}
Let the shape of $T$ with associated lattice path $L(T)$ be $\lambda=(\lambda_1,\ldots,\lambda_s)$ for some $s>0$. First we look at the case $\lambda_1>0$. Suppose $\lambda_1=\ldots=\lambda_i>\lambda_{i+1}$ for some $i$ (by convention, we identify the partition $(\lambda_1,\ldots,\lambda_s)$ with the partitions of form $(\lambda_1,\ldots,\lambda_s,0,\ldots,0)$).  Then the shape of $T'$ must be $\lambda' = (\lambda'_1,\ldots,\lambda'_{s+1})$ where $\lambda'_1=\ldots=\lambda'_{i+1}=\lambda_1$ and $\lambda'_j=\lambda_{j-1}$ for $j>i+1$. It is easy to check that labeling $L(T')$ with the labeling word $\D X$ is consistent with the definition of a multi-Catalan tableau. 

It remains to check that placing a $\beta$ in row $i+1$ of $T'$ results in a valid filling. The only possible problem we could encounter would be if the right-most box of row $i+1$ of $T'$ were an AE box that does not admit a $\beta$. However, that would mean row $i+1$ of $T'$ is an A-row, which implies row $i$ of $T$ is an A-row. Since the A's can only label the boundary edges belonging to inner corners of a tableau, that would imply row $i-1$ of $T$ has greater length, which contradicts our assumption that $\lambda_i=\lambda_1$ is the length of the longest row of $T$. Thus it would never be the case that the right-most box of row $i+1$ of $T$ is an AE box, and so placing a $\beta$ there indeed results in a valid filling.
\end{proof}

\begin{lemma}
Let $X$ be a word in $\{\D, \E, \A\}^{\ast}$, and let $T$ be a multi-Catalan tableau with $\type(T)=X\D$. A transition on the boundary of $T$ as defined above, corresponding to the PASEP transition $X\D \rightarrow X\E$ results in a valid multi-Catalan tableau.
\end{lemma}

\begin{proof}
The proof is equivalent to the proof of Lemma \ref{EX_lemma} above, except that instead we take the transpose of the tableaux and exchange the roles of $\alpha$ and $\beta$.
\end{proof}

We carefully summarize the transitions from a multi-Catalan tableau $T$ to the tableau $S$, depending on the chosen corner at which the transition occurs. We will be referring to these cases further on. First we make the following definitions. Let $T$ have size $(n,k,r)$ and let $\lambda=(\lambda_1,\ldots,\lambda_{k+r})$ be the shape of $T$. Assume that $\lambda$ has at least one non-zero part.
\begin{defn} 
We define $\lambda_R$ be the indicator that equals 1 if $T$ has a right leg, and 0 otherwise. We define $\lambda_L$ be the indicator that equals 1 if $T$ has a left leg, and 0 otherwise.
\end{defn}
\begin{defn}
We call a \emph{top-most corner} a corner such that the length of the row containing it equals $\lambda_1$. We define the indicator $\delta_{\beta}$ which equals 1 if the top-most corner contains a $\beta$, and 0 otherwise. Analogously, we call a \emph{bottom-most corner} a corner such that the length of the row containing it equals the length of the smallest non-zero row of $\lambda$. We define the indicator $\delta_{\alpha}$ which equals 1 if the bottom-most corner contains an $\alpha$, and 0 otherwise. We call a \emph{middle corner} a corner that is neither a top-most corner or a bottom-most corner.
\end{defn}
\begin{rem}\label{cases} Denote by $\pi(T \rightarrow S)$ the rate of transition from tableau $T$ to $S$ (where by rate we mean the unnormalized probability). We obtain the following cases for the transitions from $T$ to $S$.
\begin{enumerate}
\item\label{mid} For a transition at a middle corner, a top-most corner with $\delta_{\beta}=1$, or a bottom-most corner with $\delta_{\alpha}=1$, we have $\wt(S) = \wt(T)$, and $\pi(T \rightarrow S)=1$.

\item\label{top0} For a transition at a top-most corner with $\delta_{\beta}=0$ such that the length of the column containing it is greater than 1, we have $\wt(S)=\wt(T)$ and $\pi(T \rightarrow S)=1$. Then $S$ will have top-most corner that contains an $\alpha$.
\item\label{bottom0} For a transition at a bottom-most corner with $\delta_{\alpha}=0$ such that the length of the row containing it is greater than 1, we have $\wt(S)=\wt(T)$ and $\pi(T \rightarrow S)=1$. Then $S$ will have a bottom-most corner that contains a $\beta$.

\item\label{top1} For a transition at a top-most corner with $\delta_{\beta}=0$ such that the length of the column containing it is 1, we have $\wt(S)=\frac{1}{\alpha}\wt(T)$ and $\pi(T \rightarrow S)=1$.
\item\label{bottom1} For a transition at a bottom-most corner with $\delta_{\alpha}=0$ such that the length of the row containing it is 1, we have $\wt(S)=\frac{1}{\beta}\wt(T)$ and $\pi(T \rightarrow S)=1$.

\item\label{right} For a transition at a right leg, we have $\wt(S) = \alpha\wt(T)$ and $\pi(T \rightarrow S)=\alpha$. $S$ will not have a right leg, and it will have a top-most corner that contains a $\beta$.
\item\label{left} For a transition at a left leg, we have $\wt(S) = \beta\wt(T)$ and $\pi(T \rightarrow S)=\beta$. $S$ will not have a left leg, and it will have a bottom-most corner that contains an $\alpha$.
\end{enumerate}
\end{rem}
 

\begin{figure}[!ht]
\centering
\includegraphics[width=0.8\textwidth]{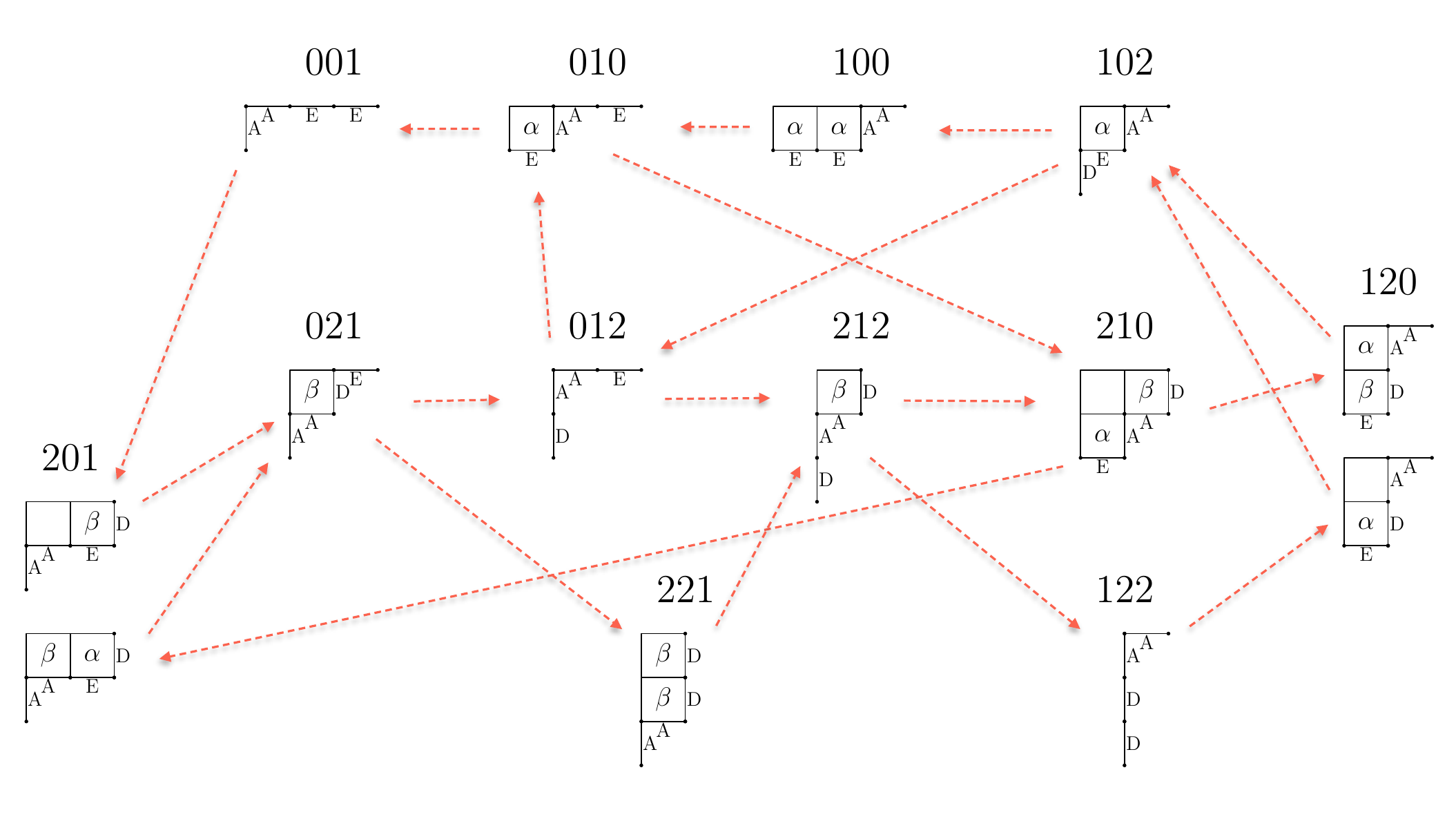}
\caption{The state diagram of a two-species PASEP at $q=0$ of size 3 and with one particle of type 1. Here the words in \{0,1,2$\}^3$ represent the states, with 0 representing a hole, 1 representing particle 1, and 2 representing particle 2.}
\noindent
\label{mc}
\end{figure}




Figure \ref{mc} shows an example of the transitions on all the states of size 3 with one particle of type 1.

\begin{thm}
The Markov chain on multi-Catalan tableaux projects to the two-species PASEP at $q=0$.
\end{thm}

\begin{proof}
Let $T$ be a multi-Catalan tableau of size $(n,k,r)$. Let $\mathcal{S}= \{S\ :\ \pi(T \rightarrow S)>0\}$ and $\mathcal{T}= \{T'\ :\ \pi(T' \rightarrow T)>0\}$. We show the following:
\begin{enumerate}[i.]
\item Detailed balance holds:
\begin{equation}\label{balance}
\wt(T) \sum_{S \in \mathcal{S}} \pi(T \rightarrow S) = \sum_{T' \in \mathcal{T}} \wt(T')\pi(T' \rightarrow T).
\end{equation}
\item For every $S \in \mathcal{S}$, we have that $\frac{1}{n+1}\pi(T \rightarrow S)$ equals the probability of the transition from the state $\type(T)$ to $\type(S)$ of the two-species PASEP.
\item For every state $X$ of the two-species PASEP and every state $Y$ for which there is a nonzero probability $\frac{p}{n+1}$ of transition from $X$, for any tableau $T$ with $\type(T)=X$, there exists a unique tableau $S$ with $\type(S)=Y$, and moreover $\pi(T \rightarrow S)=p$.
\end{enumerate}
Condition (i) implies that $\wt(T)$ is proportional to the steady state probability of $T$. Satisfying condition (ii) for all $T$ and (iii) for all states $X$ of the two-species PASEP implies that $\weight(X)$ is proportional to the steady state probability of $X$. Thus proving (i)-(iii) is sufficient to show that our Markov chain on the multi-Catalan tableaux indeed projects to the two-species PASEP at $q=0$. 

\noindent \emph{Condition (i).} Let $X=\type(T)$. First we treat the transitions going out of $T$ to $S \in \mathcal{S}$. By the construction of our Markov chain on the tableaux, it is clear that there is a transition with probability 1 for every corner, a transition with probability $\alpha$ for a right leg, and a transition with probability $\beta$ for a left leg. These transitions directly correspond to all the possible transitions out of the two-species PASEP state $X$. Suppose $X$ has $C$ corners (note that AA boxes are excluded). Thus we obtain
\begin{equation}\label{LHS}
\sum_{S \in \mathcal{S}} \pi(T \rightarrow S) = C + \alpha \delta_L + \beta \delta_R.
\end{equation}

For the transitions going into $T$ from $T' \in \mathcal{T}$, we observe that any transition from one tableau to another ends with  a corner, an edge on the right leg, or an edge on the left leg. This is because for a transition that involves either inserting into the tableau a nonempty column containing an $\alpha$ or a nonempty row containing a $\beta$, then the box containing the Greek symbol is the aforementioned corner. Otherwise, for a transition that involves inserting into the tableau an empty column or an empty row, the result is a contribution of an edge to the right leg or an edge to the left leg, respectively. Thus it is sufficient to examine the corners and the right leg and left leg of $T$ to enumerate all the possibilities for $T' \in \mathcal{T}$.  We examine the pre-image of the cases for the possible transitions going out of $T$ to obtain the following cases for $T'$.

\begin{enumerate}
\item For a middle corner, a top-most corner with $\delta_{\beta}=0$, or a bottom-most corner with $\delta_{\alpha}=0$, we have $\wt(T')=\wt(T)$ and $\pi(T' \rightarrow T)=1$. This is the inverse of Case \ref{mid} of Remark \ref{cases}. This gives a contribution of $\wt(T)(C-2+(1-\delta_{\beta})+(1-\delta_{\alpha}))$ to the right hand side (RHS) of the detailed balance equation.\footnote{Note that if $C<2$, the formulas we give have some degeneracies. However, it is easy to verify that these do not cause any problems due to cancellation of all the degenerate terms.}
\item For a top-most corner with $\delta_{\beta}=1$ and $\delta_R=0$, we have a transition involving the right-leg of $T'$, so $\wt(T')=\frac{1}{\alpha}\wt(T)$ and $\pi(T' \rightarrow T)=\alpha$. This is the inverse of Case \ref{top0} of Remark \ref{cases}. This gives a contribution of $\alpha\frac{1}{\alpha}\wt(T)\delta_{\beta}(1-\delta_R)$ to the RHS of the detailed balance equation.
\item For a bottom-most corner with $\delta_{\alpha}=1$ and $\delta_L=0$, we have a transition involving the left-leg of $T'$, so $\wt(T')=\frac{1}{\beta}\wt(T)$ and $\pi(T' \rightarrow T)=\beta$. This is the inverse of Case \ref{bottom0} of Remark \ref{cases}. This gives a contribution of $\beta\frac{1}{\beta}\wt(T)\delta_{\alpha}(1-\delta_L)$ to the RHS of the detailed balance equation.

\item For a top-most corner with $\delta_{\beta}=1$ and $\delta_R=1$, there are two possibilities. For the first, $T'$ could fall into Case \ref{top0}, meaning that $T'$ has a top-most corner containing a $\beta$, which is the usual transition with $\wt(T')=\wt(T)$. For the second possibility, $T'$ could fall into Case \ref{top1}, meaning that $T'$ has a top-most corner containing an $\alpha$ and the column containing it has length 1. In that case, $\wt(T')=\alpha\wt(T)$. In both situations, $\pi(T' \rightarrow T)=1$. We obtain a contribution of $\wt(T)\delta_{\beta}\left(\delta_R +\alpha(1-\delta_R)\right)$  to the RHS of the detailed balance equation.

\item For a bottom-most corner with $\delta_{\alpha}=1$ and $\delta_L=1$, there are two possibilities. For the first, $T'$ could fall into Case \ref{bottom0}, meaning that $T'$ has a bottom-most corner containing an $\alpha$, which is the usual transition with $\wt(T')=\wt(T)$. For the second possibility, $T'$ could fall into Case \ref{bottom1}, meaning that $T'$ has a bottom-most corner containing a $\beta$ and the row containing it has length 1. In that case, $\wt(T')=\beta\wt(T)$. In both situations, $\pi(T' \rightarrow T)=1$.  We obtain a contribution of $\wt(T)\delta_{\alpha}\left(\delta_L +\beta(1-\delta_L)\right)$ to the RHS of the detailed balance equation.

\end{enumerate}

We sum up the contributions to the RHS of the detailed balance equation to obtain
\begin{multline}\label{RHS}
\sum_{T' \in \mathcal{T}} \wt(T')\pi(T' \rightarrow T) = \wt(T) ( C - \delta_{\beta} - \delta_{\alpha} + \delta_{\beta}(1-\delta_R) + \delta_{\alpha}(1-\delta_L) \\
+ \delta_{\beta}(\delta_R +\alpha(1-\delta_R)) +\delta_{\alpha}(\delta_L +\beta(1-\delta_L)) ).
\end{multline}
We see that after simplification, Equation \ref{RHS} equals Equation \ref{LHS}, so indeed the desired Equation \ref{balance} holds for ``most'' $T$. 

It remains to check a few degenerate cases for $T$, in particular, when $\lambda(T)=(0,\ldots)$. However, those cases can only occur when $\type(T)$ contains zero $\A$'s. Thus we refer to \cite{cw_mc} for these details.


\noindent \emph{Condition (ii).} In the definition of the Markov chain, the transitions on the corners of the tableau are set to have rate 1. These transitions, which occur on DE, DA, and AE corners, precisely correspond to the transitions $X\D\E  Y \rightarrow X\E\D Y$, $X\D\A  Y \rightarrow X\A\D Y$, and $X\A\E  Y \rightarrow X\E\A Y$ on the two-species PASEP that do not involve particles hopping on and off the boundary. On the two-species PASEP, such transitions have probability $\frac{1}{n+1}$, as desired.

Similarly, the transitions involving an empty column on the east end of the tableau have rate $\alpha$, and they precisely correspond to the transition $\E X \rightarrow \D X$ of the two-species PASEP, which has probability $\frac{\alpha}{n+1}$. Analogously, the transitions involving an empty row on the south end of the tableau have rate $\beta$, and they precisely correspond to the transition $X\D \rightarrow X\D$ of the two-species PASEP, which has probability $\frac{\beta}{n+1}$.

\noindent \emph{Condition (iii).} This condition holds by the definition of the Markov chain.

\end{proof}

\begin{rem}
It would be more enlightening to interpret the multi-Catalan tableaux as certain binary trees as a natural generalization of the tree-like tableaux that are in bijection with regular Catalan tableaux (see \cite{treelike} for reference). From this perspective, it is quite easy to visualize the Markov chain on the tableaux, and prove this is a projection onto the two-species PASEP. However, we leave this description for a future paper.
\end{rem}

\section{Two-species tableaux for $q=1$}\label{section_q1}

The goal of this section is to give a combinatorial formula for the two-species PASEP for $q=1$. To this end, we define \emph{two-species alternative tableaux}. In Figure \ref{q1_rules}, we illustrate the rules given in the following definition.

\begin{figure}[h]
\centering
\includegraphics[width=\textwidth]{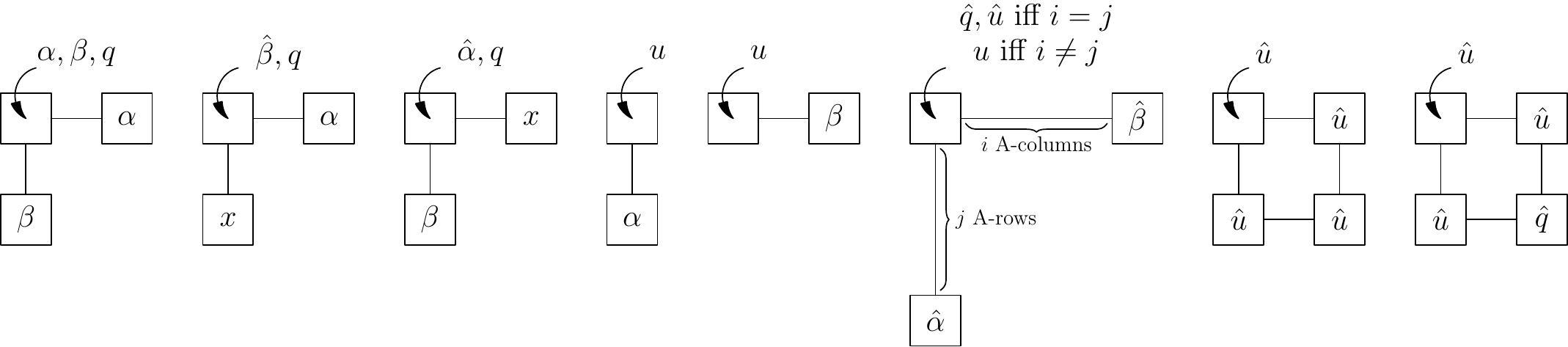}
\caption{An illustration of the rules in Definition \ref{q1_def} for the filling of a two-species alternative tableau.}
\noindent
\label{q1_rules}
\end{figure}

\begin{defn}\label{q1_def}
A \emph{two-species alternative tableau} of \emph{size} $n$ is a filling of a Young diagram of shape $(n,n-1,\ldots,1)$ with the symbols $\alpha,\hat{\alpha},\beta,\hat{\beta},q,\hat{q},u,\hat{u}$ according to the rules below.
\begin{itemize}
\item every box on the diagonal must contain an $\alpha, \beta$, or $x$.
\item a box that sees an $\alpha$ to its right and a $\beta$ below must contain an $\alpha$, $\beta$, or $q$.
\item a box that sees an $\alpha$ to its right and a $\beta$ below must contain an $\hat{\alpha}$ or $q$.
\item a box that sees an $\alpha$ to its right and a $\beta$ below must contain a $\hat{\beta}$ or $q$.
\item every box in the same column and above an $\alpha$ must contain a $u$, and every box in the same row and left of a $\beta$ must contain a $u$. 
\item for every pair of $\hat{\alpha}$'s and $\hat{\beta}$'s ($\hat{\alpha}$ left of $\hat{\beta}$) such that the number of A-rows and A-columns between them \textbf{is not} equal, put a $u$ at the intersection of the $\hat{\alpha}$ column and the $\hat{\beta}$ row. 
\item for every pair $\hat{\alpha}$'s and $\hat{\beta}$'s ($\hat{\alpha}$ left of $\hat{\beta}$) such that the number of A-rows and A-columns between them \textbf{is} equal, we must put either a $\hat{q}$ or a $\hat{u}$ at the intersection of the $\hat{\alpha}$ column and the $\hat{\beta}$ row.
\item the placement of the $\hat{q}$ and $\hat{u}$ above must satisfy that there is no instance of $\begin{smallmatrix}\hat{q}&\hat{u}\\\hat{u}&\hat{u}\end{smallmatrix}$ or $\begin{smallmatrix}\hat{q}&\hat{u}\\\hat{u}&\hat{q}\end{smallmatrix}$.
\item every other box must contain a $u$.
\end{itemize}
\end{defn}

In these fillings, the $u$'s are simply place-holders for the empty boxes, and the $\hat{u}$'s are place-holders that enforce valid placement of the $\hat{q}$'s. An easy way to construct these fillings is to place the Greek symbols and $q$'s starting from the boxes closest to the diagonal and moving inwards. Once these symbols are placed everywhere possible, we define an $\hat{\alpha}$-column to be the boxes directly above an $\hat{\alpha}$, and a $\hat{\beta}$-row to be the boxes directly to the left of a $\hat{\beta}$. We then identify the boxes that lie at the intersections of the $\hat{\alpha}$-columns and the $\hat{\beta}$-rows, and fill them appropriately with $\hat{q}$'s, $\hat{u}$'s, or $u$'s. The rest of the tableau is automatically filled with $u$'s.

\begin{figure}[h]
\centering
\includegraphics[width=0.35\textwidth]{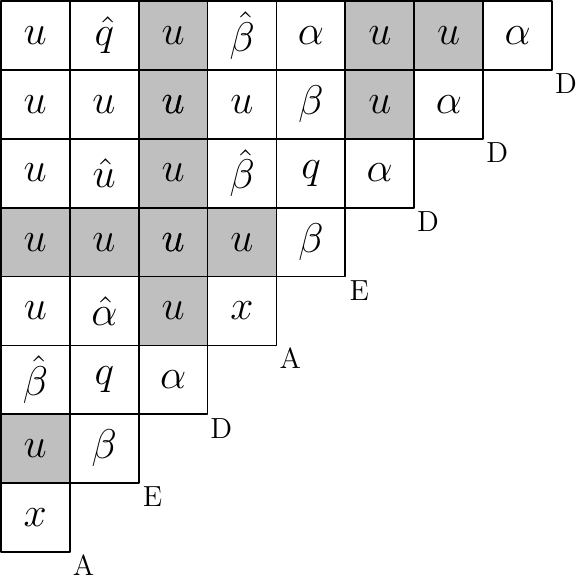}
\caption{A two-species alternative tableau of type $\D\D\D\E\A\D\E\A$ and weight $\alpha^6\beta^6$. The shaded boxes are ones that are automatically empty due to the labels on the diagonal.}
\noindent
\label{twospecies_example}
\end{figure}

\begin{defn}
The \emph{type} of the two-species alternative tableau is read off of the diagonal from top to bottom, by reading an $\alpha$ as $\D$, a $\beta$ as $\E$, and an $x$ as $\A$. The \emph{weight} of the tableau is the product of the symbols in the filling in the form of a monomial in $\alpha$ and $\beta$, where we set $u=\hat{u}=1$, $\hat{\alpha}=\alpha$, $\hat{\beta}=\beta$, and $\hat{q}=q=1$.
\end{defn}

The following conjecture is analogous to the main result of Section \ref{section_q0}, Theorem \ref{main_thm}. 

\begin{conj}
Consider the two-species PASEP at $q=1$, and let $X$ be a state represented by a word in $\{\D, \E, \A\}^n$ with precisely $r$ $\A$'s. Then the steady state probability of $X$ is
\begin{displaymath}
\Prob(X) = \frac{1}{Z^1_{n,r}}\sum_T \wt(T), 
\end{displaymath}
where the sum is over all two-species alternative tableaux $T$ such that $\type(T)=X$, and where  $Z^1_{n,r} = \sum_{T}\wt(T)$, for $T$ ranging over all two-species alternative tableaux of size $n$ whose type has exactly $r$ $\A$'s.
\end{conj}

We have verified the above using SAGE for up to $n=10$. However, so far the proof appears tedious.

\end{document}